\documentclass[12pt,twoside,reqno]{amsart}

\usepackage{version}         

\usepackage{color}

\usepackage{multirow}
\usepackage[OT1]{fontenc}
\usepackage{type1cm}
\usepackage{amssymb}
\usepackage{mathrsfs}
\usepackage[english]{babel}
\usepackage[latin1]{inputenc}
\usepackage[T1]{fontenc}
\usepackage{palatino}
\usepackage{amsfonts}
\usepackage{amsmath}
\usepackage{amssymb}
\usepackage{amsthm}
\usepackage{bbm}
\usepackage{extarrows}
\usepackage{graphicx}
\usepackage[usenames,dvipsnames]{xcolor}
\usepackage{wrapfig}
\usepackage{type1cm}
\usepackage[bf]{caption}
\usepackage{esint}
\usepackage{version}
\usepackage[colorlinks = true, citecolor = black, linkcolor = black, urlcolor = black, pdfstartview=FitH]{hyperref}
\usepackage{caption}
\usepackage{tikz}
\usepackage{bbding, wasysym}
\usepackage{enumitem}
\usepackage{float}
\usepackage[all]{xy}

\usepackage[left=2.3cm,top=3cm,right=2.3cm]{geometry}
\numberwithin{equation}{section}

\theoremstyle{plain}
\newtheorem{theorem}{Theorem}[section]

\newtheorem{corollary}[theorem]{Corollary}
\newtheorem{proposition}[theorem]{Proposition}
\newtheorem{lemma}[theorem]{Lemma}

\theoremstyle{remark}

\newtheorem*{ack}{Acknowledgement}

\theoremstyle{definition}
\newtheorem{definition}[theorem]{Definition}

\newcommand{\R}{\mathbb{R}}

\newcommand{\N}{\mathbb{N}}

\newcommand{\cL}{\mathcal{L}}

\newcommand{\cG}{\mathcal{G}}

\newcommand{\cT}{\mathcal{T}}
\newcommand{\cI}{\mathcal{I}}

\renewcommand{\epsilon}{\varepsilon}
\renewcommand{\rho}{\varrho}
\renewcommand{\phi}{\varphi}

\begin{document}

\title{Expansions in multiple bases}

\author{Yao-Qiang Li}

\address{Institut de Math\'ematiques de Jussieu - Paris Rive Gauche \\
         Sorbonne Universit\'e - Campus Pierre et Marie Curie\\
         Paris, 75005 \\
         France}

\email{yaoqiang.li@imj-prg.fr\quad yaoqiang.li@etu.upmc.fr}

\address{School of Mathematics \\
         South China University of Technology \\
         Guangzhou, 510641 \\
         P.R. China}

\email{scutyaoqiangli@gmail.com\quad scutyaoqiangli@qq.com}


\begin{abstract}
Expansion of real numbers is a basic research topic in number theory. Usually we expand real numbers in one given base. In this paper, we begin to systematically study expansions in multiple given bases in a reasonable way, which is a generalization in the sense that if all the bases are taken to be the same, we return to the classical expansions in one base. In particular, we focus on greedy, quasi-greedy, lazy, quasi-lazy and unique expansions in multiple bases.
\end{abstract}

\maketitle

\section{Introduction}

As is well known, expansion in a given base is the most common way to represent a real number. For example, expansions in base $10$ are used in our daily lives and expansions in base $2$ are used in computer systems. Expansions of real numbers in integer bases have been widely used. As a natural generalization, in 1957, R\'enyi \cite{R57} introduced expansions in non-integer bases, which attracted a lot of attention in the following decades. Until Neunh\"auserer \cite{N19} began the study of expansions in two bases recently in 2019, all expansions studied were in one base. In this paper, we begin the study of expansions in multiple bases.

Let $\N$ be the set of positive integers $\{1,2,3,\cdots\}$ and $\R$ be the set of real numbers. We recall the concept of expansions in one base first. Let $m\in\N$, $\beta\in(1,m+1]$ and $x\in\R$. A sequence $w=(w_i)_{i\ge1}\in\{0,1,\cdots,m\}^\N$ is called a \textit{$\beta$-expansion} of $x$ if
$$x=\pi_\beta(w):=\sum_{i=1}^\infty\frac{w_i}{\beta^i}.$$
It is known that $x$ has a $\beta$-expansion if and only if $x\in[0,\frac{m}{\beta-1}]$ (see for examples \cite{B14,B20,BK19,R57}).

The following question is natural to be thought of: Given $m\in\N$, $\beta_0,\beta_1,\cdots,\beta_m>1$, $x\in\R$ and $(w_i)_{i\ge1}\in\{0,1,\cdots,m\}^\N$, in which case should we say that $(w_i)_{i\ge1}$ is a $(\beta_0,\beta_1,\cdots,\beta_m)$-expansion of $x$, such that when $\beta_0,\beta_1,\cdots,\beta_m$ are taken to be the same $\beta$, we have $x=\sum_{i=1}^\infty\frac{w_i}{\beta^i}$? Proposition \ref{m-b expansion} may answer this question.

Let us give some notations first. For all $m\in\N$ and $\beta_0,\beta_1,\cdots,\beta_m>1$, we define
$$a_k:=\frac{k}{\beta_k}\quad\text{and}\quad b_k:=\frac{k}{\beta_k}+\frac{m}{\beta_k(\beta_m-1)}\quad\text{for all }k\in\{0,\cdots,m\}.$$
Note that $a_0=0$ and $b_m=\frac{m}{\beta_m-1}$. For all $m\in\N$, let
$$D_m:=\Big\{(\beta_0,\cdots,\beta_m):\beta_0,\cdots,\beta_m>1\text{ and }a_k<a_{k+1}\le b_k<b_{k+1}\text{ for all }k,\text{ }0\le k\le m-1\Big\}.$$
It is worth to note that $D_m$ is large enough to ensure that $(\underbrace{\beta,\cdots,\beta}_{m+1})\in D_m$ for all $\beta\in(1,m+1]$ and $m\in\N$, and $(\beta_0,\beta_1)\in D_1$ for all $\beta_0,\beta_1\in(1,2]$.

\begin{proposition}\label{m-b expansion}
Let $m\in\N$, $(\beta_0,\cdots,\beta_m)\in D_m$ and $x\in\R$. Then $x\in[0,\frac{m}{\beta_m-1}]$ if and only if there exists a sequence $w\in\{0,\cdots,m\}^\N$ such that
$$x=\sum_{i=1}^\infty\frac{w_i}{\beta_{w_1}\beta_{w_2}\cdots\beta_{w_i}}.$$
\end{proposition}

Thus we give the following.

\begin{definition}[Expansions in multiple bases]\label{def m-b expansions}
Let $m\in\N$, $\beta_0,\cdots,\beta_m>1$ and $x\in\R$. We say that the sequence $w\in\{0,\cdots,m\}^\N$ is a \textit{$(\beta_0,\cdots,\beta_m)$-expansion} of $x$ if
$$x=\sum_{i=1}^\infty\frac{w_i}{\beta_{w_1}\beta_{w_2}\cdots\beta_{w_i}}.$$
\end{definition}

On the one hand, it is straightforward to see that when $\beta_0,\cdots,\beta_m$ are taken to be the same $\beta$, $(\beta_0,\cdots,\beta_m)$-expansions are just $\beta$-expansions. On the other hand, we will see in Section 2 that many properties of $\beta$-expansions can be generalized to $(\beta_0,\cdots,\beta_m)$-expansions. This further confirms that our definition of expansions in multiple bases is reasonable.

Let $\sigma$ be the \textit{shift map} defined by $\sigma(w_1w_2\cdots):=w_2w_3\cdots$ for any sequence $(w_i)_{i\ge1}$. Given $\beta_0,\cdots,\beta_m>1$, for every integer $k\in\{0,\cdots,m\}$, we define the linear map $T_k$ by
$$T_k(x):=\beta_kx-k\quad\text{for }x\in\R.$$

\begin{small}
\begin{figure}[H]\label{T}
\begin{tikzpicture}
\draw[-](0,0)--(7,0);
\draw[-](0,0)--(0,7);
\draw[-](0,7)--(7,7);
\draw[-](7,0)--(7,7);
\draw[-](0,0)--(1.05,7);
\draw[-](0.875,0)--(3.85,7);
\draw[-](1.575,0)--(4.2,7);
\draw[-](3.318,0)--(7,7);
\draw[dashed](0,0)--(7,7);
\draw[-](-0.105,-0.7)--(0,0);
\draw[-](0.595,-0.7)--(0.875,0);
\draw[-](1.33,-0.7)--(1.575,0);
\draw[-](2.94,-0.7)--(3.318,0);
\draw[-](1.05,7)--(1.155,7.7);
\draw[-](3.85,7)--(4.151,7.7);
\draw[-](4.2,7)--(4.459,7.7);
\draw[-](7,7)--(7.399,7.7);
\foreach \x in {0}{\draw(\x,0)--(\x,0)node[below,outer sep=1pt,font=\footnotesize]at(\x,0){$a_0$};}
\foreach \x in {0.875}{\draw(\x,0)--(\x,0)node[below,outer sep=1pt,font=\footnotesize]at(\x,0){$a_1$};}
\foreach \x in {1.575}{\draw(\x,0)--(\x,0)node[below,outer sep=1pt,font=\footnotesize]at(\x,0){$a_2$};}
\foreach \x in {3.318}{\draw(\x,0)--(\x,0)node[below,outer sep=1pt,font=\footnotesize]at(\x,0){$a_3$};}
\foreach \x in {7}{\draw(\x,0)--(\x,0)node[below,outer sep=1pt,font=\footnotesize]at(\x,0){$\frac{3}{\beta_3-1}$};}
\foreach \x in {1.05}{\draw(\x,0)--(\x,0)node[above,outer sep=1pt,font=\footnotesize]at(\x,7){$b_0$};}
\foreach \x in {3.85}{\draw(\x,0)--(\x,0)node[above,outer sep=1pt,font=\footnotesize]at(3.8,7){$b_1$};}
\foreach \x in {4.2}{\draw(\x,0)--(\x,0)node[above,outer sep=1pt,font=\footnotesize]at(4.35,7){$b_2$};}
\foreach \x in {7}{\draw(\x,0)--(\x,0)node[above,outer sep=1pt,font=\footnotesize]at(\x,7){$b_3$};}
\foreach \x in {1.155}{\draw(\x,0)--(\x,0)node[above,outer sep=2pt]at(\x,7.7){$T_0$};}
\foreach \x in {4.151}{\draw(\x,0)--(\x,0)node[above,outer sep=2pt]at(4.151,7.7){$T_1$};}
\foreach \x in {4.459}{\draw(\x,0)--(\x,0)node[above,outer sep=2pt]at(4.709,7.7){$T_2$};}
\foreach \x in {7.399}{\draw(\x,0)--(\x,0)node[above,outer sep=2pt]at(\x,7.7){$T_3$};}
\foreach \y in {7}{\draw(0,\y)--(0,\y)node[left,outer sep=2pt,font=\footnotesize]at(0,\y){$\frac{3}{\beta_3-1}$};}
\end{tikzpicture}
\caption{The graph of $T_0,T_1,T_2$ and $T_3$ for some $(\beta_0,\beta_1,\beta_2,\beta_3)\in D_3$.}
\end{figure}
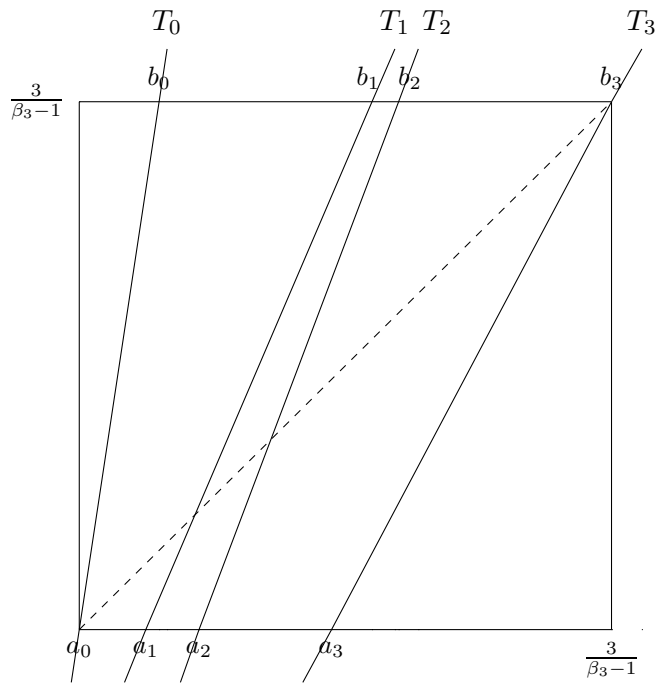
\end{small}

The main results in this paper are the following theorem and corollaries, in which $g^*$ and $l^*$ denote the quasi-greedy and quasi-lazy expansion maps respectively (see Definition \ref{tran and expan} (2) and (4)), and $\prec,\preceq,\succ,\succeq$ denote the lexicographic order. These results focus on determining greedy, lazy and unique expansions in multiple bases (see Definition \ref{tran and expan} (1) and (3)), and generalize some classical results on expansions in one base in some former well known papers.

\begin{theorem}\label{main}
Let $m\in\N$, $(\beta_0,\cdots,\beta_m)\in D_m$, $x\in[0,\frac{m}{\beta_m-1}]$, $w\in\{0,\cdots,m\}^\N$ be a $(\beta_0,\cdots,\beta_m)$-expansion of $x$ and
$$\xi_+:=\max_{0\le k\le m-1}T_k(a_{k+1}),\quad\xi_-:=\min_{0\le k\le m-1}T_k(a_{k+1}),$$
$$\eta_+:=\max_{1\le k\le m}T_k(b_{k-1}),\quad\eta_-:=\min_{1\le k\le m}T_k(b_{k-1}).$$
\begin{itemize}
\item[(1) \textcircled{\footnotesize{1}}] If $w$ is a greedy expansion, then $\sigma^nw\prec g^*(\xi_+)$ whenever $w_n<m$.
\item[\textcircled{\footnotesize{2}}] If $\sigma^nw\prec g^*(\xi_-)$ whenever $w_n<m$, then $w$ is a greedy expansion.
\item[(2) \textcircled{\footnotesize{1}}] If $w$ is a lazy expansion, then $\sigma^nw\succ l^*(\eta_-)$ whenever $w_n>0$.
\item[\textcircled{\footnotesize{2}}] If $\sigma^nw\succ l^*(\eta_+)$ whenever $w_n>0$, then $w$ is a lazy expansion.
\item[(3) \textcircled{\footnotesize{1}}] If $w$ is a unique expansion, then
$$\sigma^nw\prec g^*(\xi_+)\text{ whenever }w_n<m\quad\text{and}\quad\sigma^nw\succ l^*(\eta_-)\text{ whenever }w_n>0.$$
\item[\textcircled{\footnotesize{2}}] If
$$\sigma^nw\prec g^*(\xi_-)\text{ whenever }w_n<m\quad\text{and}\quad\sigma^nw\succ l^*(\eta_+)\text{ whenever }w_n>0,$$
then $w$ is a unique expansion.
\end{itemize}
\end{theorem}

For the case that there are at most two different bases, we get the following criteria directly from Theorem \ref{main}.

\begin{corollary}\label{cor1}
Let $\beta_0,\beta_1\in(1,2]$, $x\in[0,\frac{1}{\beta_1-1}]$ and $w\in\{0,1\}^\N$ be a $(\beta_0,\beta_1)$-expansion of $x$. Then
\newline(1) $w$ is a greedy expansion if and only if $\sigma^nw\prec g^*(\frac{\beta_0}{\beta_1})$ whenever $w_n=0$;
\newline(2) $w$ is a lazy expansion if and only if $\sigma^nw\succ l^*(\frac{\beta_1}{\beta_0(\beta_1-1)}-1)$ whenever $w_n=1$;
\newline(3) $w$ is a unique expansion if and only if
$$\sigma^nw\prec g^*(\frac{\beta_0}{\beta_1})\text{ whenever }w_n=0\quad\text{and}\quad\sigma^nw\succ l^*(\frac{\beta_1}{\beta_0(\beta_1-1)}-1)\text{ whenever }w_n=1.$$
\end{corollary}

The following corollary provide some ways to determine whether an expansion is greedy, lazy or unique by the quasi-greedy expansion of $1$ and the quasi-lazy expansion of $\frac{m}{\beta_m-1}-1$.

\begin{corollary}\label{cor2}
Let $m\in\N$, $(\beta_0,\cdots,\beta_m)\in D_m$, $x\in[0,\frac{m}{\beta_m-1}]$ and $w\in\{0,\cdots,m\}^\N$ be a $(\beta_0,\cdots,\beta_m)$-expansion of $x$.
\begin{itemize}
\item[(1) \textcircled{\footnotesize{1}}] Suppose $\beta_0\le\beta_1\le\cdots\le\beta_m$. If $w$ is a greedy expansion, then $\sigma^nw\prec g^*(1)$ whenever $w_n<m$.
\item[\textcircled{\footnotesize{2}}] Suppose $\beta_0\ge\beta_1\ge\cdots\ge\beta_m$. If $\sigma^nw\prec g^*(1)$ whenever $w_n<m$, then $w$ is a greedy expansion.
\end{itemize}
\begin{itemize}
\item[(2) \textcircled{\footnotesize{1}}] Suppose $\beta_0\le\beta_1\le\cdots\le\beta_m$. If $w$ is a lazy expansion, then $\sigma^nw\succ l^*(\frac{m}{\beta_m-1}-1)$ whenever $w_n>0$.
\item[\textcircled{\footnotesize{2}}] Suppose $\beta_0\ge\beta_1\ge\cdots\ge\beta_m$. If $\sigma^nw\succ l^*(\frac{m}{\beta_m-1}-1)$ whenever $w_n>0$, then $w$ is a lazy expansion.
\end{itemize}
\begin{itemize}
\item[(3) \textcircled{\footnotesize{1}}] Suppose $\beta_0\le\beta_1\le\cdots\le\beta_m$. If $w$ is a unique expansion, then
    $$\sigma^nw\prec g^*(1)\quad\text{whenever }w_n<m\quad\text{and}\quad\sigma^nw\succ l^*(\frac{m}{\beta_m-1}-1)\quad\text{whenever }w_n>0.$$
\item[\textcircled{\footnotesize{2}}] Suppose $\beta_0\ge\beta_1\ge\cdots\ge\beta_m$. If
$$\sigma^nw\prec g^*(1)\quad\text{whenever }w_n<m\quad\text{and}\quad\sigma^nw\succ l^*(\frac{m}{\beta_m-1}-1)\quad\text{whenever }w_n>0,$$
then $w$ is a unique expansion.
\end{itemize}
\end{corollary}

Take $\beta_0,\cdots,\beta_m$ to be the same $\beta$. By Corollary \ref{cor2}, Proposition \ref{symmetric}, Lemma \ref{lemma} and Proposition \ref{commutative}, we get the following corollary, in which $\overline{k}:=m-k$ for all $k\in\{0,\cdots,m\}$ and $\overline{w}:=(\overline{w_i})_{i\ge1}$ for all $w=(w_i)_{i\ge1}\in\{0,\cdots,m\}^\N$.

\begin{corollary}\label{cor3}
Let $m\in\N$, $\beta\in(1,m+1]$, $x\in[0,\frac{m}{\beta-1}]$ and $w\in\{0,\cdots,m\}^\N$ be a $\beta$-expansion of $x$. Then:
\begin{itemize}
\item[(1) \textcircled{\footnotesize{1}}] $w$ is a greedy expansion if and only if $\sigma^nw\prec g^*(1)$ whenever $w_n<m$;
\item[\textcircled{\footnotesize{2}}] $w$ is a lazy expansion if and only if $\sigma^nw\succ\overline{g^*(1)}$ whenever $w_n>0$;
\item[\textcircled{\footnotesize{3}}] $w$ is a unique expansion if and only if
$$\sigma^nw\prec g^*(1)\text{ whenever }w_n<m\quad\text{and}\quad\sigma^nw\succ\overline{g^*(1)}\text{ whenever }w_n>0.$$
\end{itemize}
\begin{itemize}
\item[(2) \textcircled{\footnotesize{1}}] $0\le x<1$ and $w$ is a greedy expansion if and only if $\sigma^nw\prec g^*(1)$ for all $n\ge0$;
\item[\textcircled{\footnotesize{2}}] $\frac{m}{\beta-1}-1<x\le\frac{m}{\beta-1}$ and $w$ is a lazy expansion if and only if $\sigma^nw\succ\overline{g^*(1)}$ for all $n\ge0$;
\item[\textcircled{\footnotesize{3}}] $\frac{m}{\beta-1}-1<x<1$ and $w$ is a unique expansion if and only if
    $$\overline{g^*(1)}\prec\sigma^nw\prec g^*(1)\quad\text{for all }n\ge0.$$
\end{itemize}
\end{corollary}

This corollary recovers some classical results. See for examples \cite[Theorem 1.1]{DK09}, \cite[Lemma 4]{GS01} and \cite[Theorem 3]{P60}. See also \cite[Theorem 2.1]{AFH07} and \cite[Lemma 2.11]{S03b}).

Many former papers on $\beta$-expansions are restricted to bases belonging to $(m,m+1]$ or expansion sequences belonging to $\{0,1,\cdots,\lceil\beta\rceil-1\}^\N$ (see for examples \cite{D09,DK09,K12}), where $\lceil\beta\rceil$ denotes the smallest integer no less than $\beta$. Even if all $\beta_0,\cdots,\beta_m$ are taken to be the same $\beta$ throughout this paper, we are working under a more general framework: bases belonging to $(1,m+1]$ and expansion sequences belonging to $\{0,1,\cdots,m\}^\N$ (for examples Corollary \ref{cor3} and Proposition \ref{symmetric}. See also \cite{B14,DKL16,KKLL19}).

This paper is organized as follows. In Section 2, we give some notations and study some basic properties of greedy, quasi-greedy, lazy and quasi-lazy expansions in multiple-bases. Section 3 is devoted to the proof our main results. In the last section, we present some further questions.

\section{Greedy, quasi-greedy, lazy and quasi-lazy expansions in multiple bases}

Let $m\in\N$ and $\beta_0,\cdots,\beta_m>1$. We define the \textit{projection} $\pi_{\beta_0,\cdots,\beta_m}$ by
$$\pi_{\beta_0,\cdots,\beta_m}(w_1\cdots w_n):=\sum_{i=1}^n\frac{w_i}{\beta_{w_1}\beta_{w_2}\cdots\beta_{w_i}}$$
for $w_1\cdots w_n\in\{0,\cdots,m\}^n$ and $n\in\N$, and
$$\pi_{\beta_0,\cdots,\beta_m}(w)=\pi_{\beta_0,\cdots,\beta_m}(w_1w_2\cdots):=\lim_{n\to\infty}\pi_{\beta_0,\cdots,\beta_m}(w_1\cdots w_n)=\sum_{i=1}^\infty\frac{w_i}{\beta_{w_1}\beta_{w_2}\cdots\beta_{w_i}}$$
for $w=(w_i)_{i\ge1}\in\{0,\cdots,m\}^n$. When $\beta_0,\cdots,\beta_m$ are understood from the context, we usually use $\pi$ instead of $\pi_{\beta_0,\cdots,\beta_m}$ for simplification.

\begin{definition}[Transformations and expansions]\label{tran and expan} Let $m\in\N$ and $(\beta_0,\cdots,\beta_m)\in D_m$.
\begin{itemize}
\item[(1)] The \textit{greedy $(\beta_0,\cdots,\beta_m)$-transformation} $G_{\beta_0,\cdots,\beta_m}:[0,\frac{m}{\beta_m-1}]\to[0,\frac{m}{\beta_m-1}]$ is defined by
$$x\mapsto G_{\beta_0,\cdots,\beta_m}x:=\left\{\begin{array}{ll}
T_kx & \text{if }x\in[a_k,a_{k+1})\text{ for some }k\in\{0,\cdots,m-1\};\\
T_mx & \text{if }x\in[a_m,b_m].
\end{array}\right.$$
For all $x\in[0,\frac{m}{\beta_m-1}]$ and $n\in\N$, let
$$g_n(x;\beta_0,\cdots,\beta_m):=\left\{\begin{array}{ll}
k& \text{if }G_{\beta_0,\cdots,\beta_m}^{n-1}x\in[a_k,a_{k+1})\text{ for some }k\in\{0,\cdots,m-1\};\\
m & \text{if }G_{\beta_0,\cdots,\beta_m}^{n-1}x\in[a_m,b_m].
\end{array}\right.$$
We call the sequence $g(x;\beta_0,\cdots,\beta_m):=(g_n(x;\beta_0,\cdots,\beta_m))_{n\ge1}$ the \textit{greedy $(\beta_0,\cdots,\beta_m)$-expansion} of $x$.
\item[(2)] The \textit{quasi-greedy $(\beta_0,\cdots,\beta_m)$-transformation} $G^*_{\beta_0,\cdots,\beta_m}:[0,\frac{m}{\beta_m-1}]\to[0,\frac{m}{\beta_m-1}]$ is defined by
$$x\mapsto G^*_{\beta_0,\cdots,\beta_m}x:=\left\{\begin{array}{ll}
T_0x & \text{if }x\in[0,a_1];\\
T_kx & \text{if }x\in(a_k,a_{k+1}]\text{ for some }k\in\{1,\cdots,m-1\};\\
T_mx & \text{if }x\in(a_m,b_m].
\end{array}\right.$$
For all $x\in[0,\frac{m}{\beta_m-1}]$ and $n\in\N$, let
$$g^*_n(x;\beta_0,\cdots,\beta_m):=\left\{\begin{array}{ll}
0& \text{if }(G^*_{\beta_0,\cdots,\beta_m})^{n-1}x\in[0,a_1];\\
k& \text{if }(G^*_{\beta_0,\cdots,\beta_m})^{n-1}x\in(a_k,a_{k+1}]\text{ for some }k\in\{1,\cdots,m-1\};\\
m & \text{if }(G^*_{\beta_0,\cdots,\beta_m})^{n-1}x\in(a_m,b_m].
\end{array}\right.$$
We call the sequence $g^*(x;\beta_0,\cdots,\beta_m):=(g^*_n(x;\beta_0,\cdots,\beta_m))_{n\ge1}$ the \textit{quasi-greedy $(\beta_0,\cdots,\beta_m)$-expansion} of $x$.
\item[(3)] The \textit{lazy $(\beta_0,\cdots,\beta_m)$-transformation} $L_{\beta_0,\cdots,\beta_m}:[0,\frac{m}{\beta_m-1}]\to[0,\frac{m}{\beta_m-1}]$ is defined by
$$x\mapsto L_{\beta_0,\cdots,\beta_m}x:=\left\{\begin{array}{ll}
T_0x & \text{if }x\in[0,b_0];\\
T_kx& \text{if }x\in(b_{k-1},b_k]\text{ for some }k\in\{1,\cdots,m\}.
\end{array}\right.$$
For all $x\in[0,\frac{m}{\beta_m-1}]$ and $n\in\N$, let
$$l_n(x;\beta_0,\cdots,\beta_m):=\left\{\begin{array}{ll}
0& \text{if }L_{\beta_0,\cdots,\beta_m}^{n-1}x\in[0,b_0];\\
k& \text{if }L_{\beta_0,\cdots,\beta_m}^{n-1}x\in(b_{k-1},b_k]\text{ for some }k\in\{1,\cdots,m\}.
\end{array}\right.$$
We call the sequence $l(x;\beta_0,\cdots,\beta_m):=(l_n(x;\beta_0,\cdots,\beta_m))_{n\ge1}$ the \textit{lazy $(\beta_0,\cdots,\beta_m)$-expansion} of $x$.
\item[(4)] The \textit{quasi-lazy $(\beta_0,\cdots,\beta_m)$-transformation} $L^*_{\beta_0,\cdots,\beta_m}:[0,\frac{m}{\beta_m-1}]\to[0,\frac{m}{\beta_m-1}]$ is defined by
$$x\mapsto L^*_{\beta_0,\cdots,\beta_m}x:=\left\{\begin{array}{ll}
T_0x & \text{if }x\in[0,b_0);\\
T_kx & \text{if }x\in[b_{k-1},b_k)\text{ for some }k\in\{1,\cdots,m-1\};\\
T_mx & \text{if }x\in[b_{m-1},b_m].
\end{array}\right.$$
For all $x\in[0,\frac{m}{\beta_m-1}]$ and $n\in\N$, let
$$l^*_n(x;\beta_0,\cdots,\beta_m):=\left\{\begin{array}{ll}
0& \text{if }(L^*_{\beta_0,\cdots,\beta_m})^{n-1}x\in[0,b_0);\\
k& \text{if }(L^*_{\beta_0,\cdots,\beta_m})^{n-1}x\in[b_{k-1},b_k)\text{ for some }k\in\{1,\cdots,m-1\};\\
m & \text{if }(L^*_{\beta_0,\cdots,\beta_m})^{n-1}x\in[b_{m-1},b_m].
\end{array}\right.$$
We call the sequence $l^*(x;\beta_0,\cdots,\beta_m):=(l^*_n(x;\beta_0,\cdots,\beta_m))_{n\ge1}$ the \textit{quasi-lazy $(\beta_0,\cdots,\beta_m)$-expansion} of $x$.
\end{itemize}
Generally, let $\cI_{\beta_0,\cdots,\beta_m}$ be the set of tuples $(I_0,\cdots,I_m)$ which satisfy
$$I_0\in\Big\{[0,c_1],[0,c_1)\Big\},$$
$$I_k\in\Big\{[c_k,c_{k+1}],[c_k,c_{k+1}),(c_k,c_{k+1}],(c_k,c_{k+1})\Big\}$$
for all $k\in\{1,\cdots,m-1\}$, and
$$I_m\in\Big\{[c_m,\frac{m}{\beta_m-1}],(c_m,\frac{m}{\beta_m-1}]\Big\},$$
where
$$c_k\in[a_k,b_{k-1}]\quad\text{for all }k\in\{1,\cdots,m\}$$
such that $c_1<c_2<\cdots<c_m$, $I_0\cup I_1\cup\cdots\cup I_m=[0,\frac{m}{\beta_m-1}]$ and $I_0,I_1,\cdots,I_m$ are all disjoint. For any $(I_0,\cdots,I_m)\in\cI_{\beta_0,\cdots,\beta_m}$, we define the \textit{$(I_0,\cdots,I_m)$-$(\beta_0,\cdots,\beta_m)$-transformation} $T_{\beta_0,\cdots,\beta_m}^{I_0,\cdots,I_m}:[0,\frac{m}{\beta_m-1}]\to[0,\frac{m}{\beta_m-1}]$ by
$$T_{\beta_0,\cdots,\beta_m}^{I_0,\cdots,I_m}(x):=T_k(x)\quad\text{for }x\in I_k\text{ where }k\in\{0,\cdots,m\}.$$
For all $x\in[0,\frac{m}{\beta_m-1}]$ and $n\in\N$, let
$$t_n(x;\beta_0,\cdots,\beta_m;I_0,\cdots,I_m):=k\quad\text{if }(T_{\beta_0,\cdots,\beta_m}^{I_0,\cdots,I_m})^{n-1}x\in I_k\text{ where }k\in\{0,\cdots,m\}.$$
We call the sequence $t(x;\beta_0,\cdots,\beta_m;I_0,\cdots,I_m):=(t_n(x;\beta_0,\cdots,\beta_m;I_0,\cdots,I_m))_{n\ge1}$ the \textit{$(I_0,\cdots,I_m)$-$(\beta_0,\cdots,\beta_m)$-expansion} of $x$.
\end{definition}

It is straightforward to see that greedy, quasi-greedy, lazy and quasi-lazy $(\beta_0,\cdots,\beta_m)$-transformations/expansions are special cases of some $(I_0,\cdots,I_m)$-$(\beta_0,\cdots,\beta_m)$-transfor-
mations/expansions. For simplification, on the one hand, if $\beta_0,\cdots,\beta_m$ are understood from the context, we use $G,G^*,L,L^*,g(x),g^*(x),l(x)$ and $l^*(x)$ instead of $G_{\beta_0,\cdots,\beta_m}$, $G^*_{\beta_0,\cdots,\beta_m}$, $L_{\beta_0,\cdots,\beta_m}$, $L^*_{\beta_0,\cdots,\beta_m}$, $g(x;\beta_0,\cdots,\beta_0)$, $g^*(x;\beta_0,\cdots,\beta_0)$, $l(x;\beta_0,\cdots,\beta_0)$ and $l^*(x;\beta_0,\cdots,\beta_0)$ respectively, and if $x$ is also understood, we use $g_n,g^*_n,l_n$ and $l^*_n$ instead of $g_n(x;\beta_0,\cdots,\beta_m)$, $g^*_n(x;\beta_0,\cdots,\beta_m)$, $l_n(x;\beta_0,\cdots,\beta_m)$ and $l^*_n(x;\beta_0,\cdots,\beta_m)$ respectively for all $n\in\N$; on the other hand, if $\beta_0,\cdots,\beta_m$ and $I_0,\cdots,I_m$ are understood, we use $T$ and $t(x)$ instead of $T_{\beta_0,\cdots,\beta_m}^{I_0,\cdots,I_m}$ and $t(x;\beta_0,\cdots,\beta_m;I_0,\cdots,I_m)$ respectively, and if $x$ is also understood, we use $t_n$ instead of $t_n(x;\beta_0,\cdots,\beta_m;I_0,\cdots,I_m)$ for all $n\in\N$.

For the case of a single base, greedy $\beta$-transformations and expansions were studied in many former papers \cite{B89,BB04,BW14,FW12,FS92,S97,S80}), lazy $\beta$-transformations and expansions can be found in \cite{DD07,DK02,DK16,EJK90,KS12}, and quasi-greedy $\beta$-expansions were introduced in \cite{KL07,L06,P05}.

In Proposition \ref{expansions}, we will see that the above definition really give $(\beta_0,\cdots,\beta_m)$-expansions coincide with Definition \ref{def m-b expansions}. First we prove the following useful lemma.

\begin{lemma}\label{iteration}
Let $m\in\N$, $(\beta_0,\cdots,\beta_m)\in D_m$ and $x\in[0,\frac{m}{\beta_m-1}]$. If $(I_0,\cdots,I_m)\in\cI_{\beta_0,\cdots,\beta_m}$, then for all $n\in\N$, we have
$$x=\pi(t_1\cdots t_n)+\frac{T^nx}{\beta_{t_1}\cdots\beta_{t_n}}.$$
In particular, for all $n\in\N$, we have
$$\begin{aligned}
x&=\pi(g_1\cdots g_n)+\frac{G^nx}{\beta_{g_1}\cdots\beta_{g_n}}=\pi(g^*_1\cdots g^*_n)+\frac{(G^*)^nx}{\beta_{g^*_1}\cdots\beta_{g^*_n}}\\
&=\pi(l_1\cdots l_n)+\frac{L^nx}{\beta_{l_1}\cdots\beta_{l_n}}=\pi(l^*_1\cdots l^*_n)+\frac{(L^*)^nx}{\beta_{l^*_1}\cdots\beta_{l^*_n}}.
\end{aligned}$$
\end{lemma}
\begin{proof} (By induction) Let $k\in\{0,\cdots,m\}$ such that $x\in I_k$. Then $t_1=k$, $Tx=T_kx=\beta_kx-k$ and we have
$$\pi(t_1)+\frac{Tx}{\beta_{t_1}}=\frac{t_1+Tx}{\beta_{t_1}}=\frac{\beta_kx}{\beta_k}=x.$$
Suppose that the conclusion is true for some $n\in\N$, we prove that it is also true for $n+1$ as follows. In fact, we have
$$\begin{aligned}
\pi(t_1\cdots t_{n+1})+\frac{T^{n+1}x}{\beta_{t_1}\cdots\beta_{t_{n+1}}}&=\pi(t_1\cdots t_n)+\frac{t_{n+1}+T^{n+1}x}{\beta_{t_1}\cdots\beta_{t_{n+1}}}\\
&\overset{(\star)}{=}\pi(t_1\cdots t_n)+\frac{\beta_{t_{n+1}}T^{n}x}{\beta_{t_1}\cdots\beta_{t_{n+1}}}\\
&=x,
\end{aligned}$$
where the last equality follows from the inductive hypothesis and $(\star)$ can be proved as follows. Let $k\in\{0,\cdots,m\}$ such that $T^nx\in I_k$. Then $t_{n+1}=k$ and
$$t_{n+1}+T^{n+1}x=t_{n+1}+T_k(T^nx)=k+(\beta_kT^nx-k)=\beta_{t_{n+1}}T^nx.$$
\end{proof}

\begin{proposition}\label{expansions} Let $m\in\N$, $(\beta_0,\cdots,\beta_m)\in D_m$ and $x\in[0,\frac{m}{\beta_m-1}]$. If $(I_0,\cdots,I_m)\in\cI_{\beta_0,\cdots,\beta_m}$, then the $(I_0,\cdots,I_m)$-$(\beta_0,\cdots,\beta_m)$-expansion of $x$ is a $(\beta_0,\cdots,\beta_m)$-expansion of $x$, i.e.,
$$x=\pi(t(x)),$$
and for all $n\in\N$ we have
$$T^nx=\pi(t_{n+1}t_{n+2}\cdots).$$
In particular, greedy, quasi-greedy, lazy and quasi-lazy $(\beta_0,\cdots,\beta_m)$-expansions of $x$ are all $(\beta_0,\cdots,\beta_m)$-expansions of $x$, i.e.,
$$x=\pi(g(x))=\pi(g^*(x))=\pi(l(x))=\pi(l^*(x)),$$
and for all $n\in\N$ we have
$$G^nx=\pi(g_{n+1}g_{n+2}\cdots),\quad(G^*)^nx=\pi(g^*_{n+1}g^*_{n+2}\cdots),$$
$$L^nx=\pi(l_{n+1}l_{n+2}\cdots),\quad(L^*)^nx=\pi(l^*_{n+1}l^*_{n+2}\cdots).$$
\end{proposition}
\begin{proof} By Lemma \ref{iteration} and
$$\frac{T^nx}{\beta_{t_1}\cdots\beta_{t_n}}\le\frac{\frac{m}{\beta_m-1}}{(\min\{\beta_0,\cdots,\beta_m\})^n}\to 0$$
as $n\to\infty$, we get $x=\lim_{n\to\infty}\pi(t_1\cdots t_n)=\pi(t(x))$. That is,
$$x=\pi(t_1\cdots t_n)+\frac{\pi(t_{n+1}t_{n+2}\cdots)}{\beta_{t_1}\cdots\beta_{t_n}}.$$
It follows from Lemma \ref{iteration} that $T^nx=\pi(t_{n+1}t_{n+2}\cdots)$.
\end{proof}

Greedy, quasi-greedy, lazy and quasi-lazy expansions are not necessarily identical. A real number may have many different expansions even in one given base (see for examples \cite{EHJ91,EJK90,S03a}).

\begin{proof}[Proof of Proposition \ref{m-b expansion}] $\boxed{\Rightarrow}$ follows from Proposition \ref{expansions}.
\newline$\boxed{\Leftarrow}$ Let $w\in\{0,\cdots,m\}^\N$ and $x=\pi(w)$. It suffices to prove $x\le\frac{m}{\beta_m-1}$ in the following. (By contradiction) We assume $x>\frac{m}{\beta_m-1}$.
\begin{itemize}
\item[(1)] Prove that for all $v\in\{0,\cdots,m\}^\N$ and $n\in\N$, we have
$$T_{v_n}\circ\cdots\circ T_{v_1}x>\cdots>T_{v_2}\circ T_{v_1}x>T_{v_1}x>x.$$
Let $k\in\{0,\cdots,m-1\}$, by $(\beta_0,\cdots,\beta_m)\in D_m$, we get
$$\frac{k}{\beta_k}+\frac{m}{\beta_k(\beta_m-1)}=b_k<b_{k+1}<\cdots<b_m=\frac{m}{\beta_m-1},$$
which implies $\frac{k}{\beta_k-1}<\frac{m}{\beta_m-1}$. Thus for all $y>\frac{m}{\beta_m-1}$ and $k\in\{0,\cdots,m\}$, we have $y>\frac{k}{\beta_k-1}$, i.e., $T_ky>y$. Then we perform the maps $T_{v_1},\cdots,T_{v_n}$ to $x$ one by one to get the conclusion.
\item[(2)] Let $s\in\{0,\cdots,m\}$ such that $T_sx=\min_{0\le k\le m}T_kx$. For all $n\in\N$, we prove
    $$T_{w_{n+1}}\circ\cdots\circ T_{w_1}x-T_{w_n}\circ\cdots\circ T_{w_1}x>T_sx-x.$$
    In fact, it suffices to prove
    $$T_{w_{n+1}}\circ\cdots\circ T_{w_1}x-T_{w_n}\circ\cdots\circ T_{w_1}x>T_{w_{n+1}}x-x.$$
    This follows from
$$\begin{aligned}
T_{w_{n+1}}\circ T_{w_n}\circ\cdots\circ T_{w_1}x-T_{w_{n+1}}x&=(\beta_{w_{n+1}}T_{w_n}\circ\cdots\circ T_{w_1}x-w_{n+1})-(\beta_{w_{n+1}}x-w_{n+1})\\
&=\beta_{w_{n+1}}(T_{w_n}\circ\cdots\circ T_{w_1}x-x)\\
&>T_{w_n}\circ\cdots\circ T_{w_1}x-x
\end{aligned}$$
where the last inequality follows from $\beta_{w_{n+1}}>1$ and $T_{w_n}\circ\cdots\circ T_{w_1}x-x>0$ (by (1)).
\item[(3)] Deduce a contradiction.
\newline On the one hand, for all $n\in\N$, we have
$$\begin{aligned}
T_{w_n}\circ\cdots\circ T_{w_1}x=&\indent(T_{w_n}\circ\cdots\circ T_{w_1}x-T_{w_{n-1}}\circ\cdots\circ T_{w_1}x)\\
&+(T_{w_{n-1}}\circ\cdots\circ T_{w_1}x-T_{w_{n-2}}\circ\cdots\circ T_{w_1}x)\\
&+\cdots\\
&+(T_{w_2}\circ T_{w_1}x-T_{w_1}x)\\
&+(T_{w_1}x-x)+x\\
\overset{\text{by }(2)}{\ge}&n(T_sx-x)+x,
\end{aligned}$$
where $T_sx-x>0$ by (1). This implies $T_{w_n}\circ\cdots\circ T_{w_1}x\to\infty$ as $n\to\infty$.
\newline On the other hand, by
$$x=\sum_{i=1}^\infty\frac{w_i}{\beta_{w_1}\cdots\beta_{w_i}},$$
we get
$$T_{w_1}x=\sum_{i=2}^\infty\frac{w_i}{\beta_{w_2}\cdots\beta_{w_i}},$$
$$T_{w_2}\circ T_{w_1}x=\sum_{i=3}^\infty\frac{w_i}{\beta_{w_3}\cdots\beta_{w_i}},$$
$$\cdots,$$
and then for all $n\in\N$,
$$\begin{aligned}
T_{w_n}\circ\cdots\circ T_{w_1}x&=\sum_{i=n+1}^\infty\frac{w_i}{\beta_{w_{n+1}}\cdots\beta_{w_i}}\\
&\le\sum_{i=n+1}^\infty\frac{m}{(\min\{\beta_0,\cdots,\beta_m\})^{i-n}}\\
&=\frac{m}{\min\{\beta_0,\cdots,\beta_m\}-1}<\infty,
\end{aligned}$$
which contradicts $T_{w_n}\circ\cdots\circ T_{w_1}x\to\infty$ as $n\to\infty$.
\end{itemize}
\end{proof}

We should keep the following lemma in mind.

\begin{lemma}\label{iff0}
Let $m\in\N$, $(\beta_0,\cdots,\beta_m)\in D_m$ and $w\in\{0,\cdots,m\}^\N$. Then $w=m^\infty$ if and only if $\pi(w)=\frac{m}{\beta_m-1}$.
\end{lemma}
\begin{proof}
$\boxed{\Rightarrow}$ is obvious.
\newline$\boxed{\Leftarrow}$ (By contradiction) Suppose $w\neq m^\infty$ and
\begin{eqnarray}\label{suppose}
\sum_{i=1}^\infty\frac{w_i}{\beta_{w_1}\cdots\beta_{w_i}}=\frac{m}{\beta_m-1}.
\end{eqnarray}
Then there exists $k\in\N$ such that $w_1\cdots w_{k-1}=m^{k-1}$ and $w_k<m$. By applying $T_m^{k-1}$ to (\ref{suppose}), we get
$$\frac{w_k}{\beta_{w_k}}+\sum_{i=k+1}^\infty\frac{w_i}{\beta_{w_k}\cdots\beta_{w_i}}=\frac{m}{\beta_m-1}.$$
It follows from applying $T_{w_k}$ to the above equality that
\begin{eqnarray}\label{get1}
\sum_{i=1}^\infty\frac{w_{k+i}}{\beta_{w_{k+1}}\cdots\beta_{w_{k+i}}}=\frac{m\beta_{w_k}}{\beta_m-1}-w_k.
\end{eqnarray}
On the one hand, by Proposition \ref{m-b expansion} we know
\begin{eqnarray}\label{get2}
\sum_{i=1}^\infty\frac{w_{k+i}}{\beta_{w_{k+1}}\cdots\beta_{w_{k+i}}}\le\frac{m}{\beta_m-1}.
\end{eqnarray}
On the other hand, by $(\beta_0,\cdots,\beta_m)\in D_m$ and $w_k<m$, we get
$$\frac{w_k}{\beta_{w_k}}+\frac{m}{\beta_{w_k}(\beta_m-1)}=b_{w_k}<b_{w_k+1}<\cdots<b_m=\frac{m}{\beta_m-1},$$
which implies
$$\frac{m\beta_{w_k}}{\beta_m-1}-w_k>\frac{m}{\beta_m-1}.$$
This contradicts (\ref{get1}) and (\ref{get2}).
\end{proof}

The following useful criteria generalize \cite[Lemma 1]{EJK90}.

\begin{proposition}[Basic criteria of greedy, quasi-greedy, lazy and quasi-lazy expansions]\label{iff}\indent
\newline Let $m\in\N$, $(\beta_0,\cdots,\beta_m)\in D_m$, $x\in[0,\frac{m}{\beta_m-1}]$ and $w\in\{0,\cdots,m\}^\N$ be a $(\beta_0,\cdots,\beta_m)$-expansion of $x$.
\newline(1) $w$ is the greedy expansion if and only if
$$\pi(w_nw_{n+1}\cdots)<a_{w_n+1}\quad\text{whenever }w_n<m.$$
(2) When $x\neq0$, $w$ is the quasi-greedy expansion if and only if it does not end with $0^\infty$ and
$$\pi(w_nw_{n+1}\cdots)\le a_{w_n+1}\quad\text{whenever }w_n<m.$$
(3) $w$ is the lazy expansion if and only if
$$\pi(w_nw_{n+1}\cdots)>b_{w_n-1}\quad\text{whenever }w_n>0.$$
(4) When $x\neq\frac{m}{\beta_m-1}$, $w$ is the quasi-lazy expansion if and only if it does not end with $m^\infty$ and
$$\pi(w_nw_{n+1}\cdots)\ge b_{w_n-1}\quad\text{whenever }w_n>0.$$
\end{proposition}
\begin{proof} (1) $\boxed{\Rightarrow}$ Suppose that $w$ is the greedy $(\beta_0,\cdots,\beta_m)$-expansion of $x$, i.e., $(w_i)_{i\ge1}=(g_i)_{i\ge1}$, and suppose $w_n<m$. By $g_n=w_n$ and the definition of $g_n$, we get $G^{n-1}x<a_{w_n+1}$. It follows from Proposition \ref{expansions} that $\pi(g_ng_{n+1}\cdots)<a_{w_n+1}$. Thus $\pi(w_nw_{n+1}\cdots)<a_{w_n+1}$.
\newline$\boxed{\Leftarrow}$ We prove $(w_i)_{i\ge1}=(g_i)_{i\ge1}$ by induction. Recall that
$$g_1:=\left\{\begin{array}{ll}
k & \text{if }x\in[a_k,a_{k+1})\text{ for some }k\in\{0,\cdots,m-1\}\\
m & \text{if }x\in[a_m,b_m]
\end{array}\right.$$
and $(w_i)_{i\ge1}$ is a $(\beta_0,\cdots,\beta_m)$-expansion of $x$, which implies $x\ge a_{w_1}$.
\begin{itemize}
\item[i)] If $w_1=m$, then $x\ge a_m$, which implies $g_1=m=w_1$.
\item[ii)] If $w_1<m$, by condition $\pi(w_1w_2\cdots)<a_{w_1+1}$ we get $x<a_{w_1+1}$. It follows from $x\ge a_{w_1}$ that $g_1=w_1$.
\end{itemize}
Suppose $w_1\cdots w_{n-1}=g_1\cdots g_{n-1}$ for some $n\ge2$. We need to prove $w_n=g_n$ in the following. Recall
$$g_n:=\left\{\begin{array}{ll}
k & \text{if }G^{n-1}x\in[a_k,a_{k+1})\text{ for some }k\in\{0,\cdots,m-1\};\\
m & \text{if }G^{n-1}x\in[a_m,b_m].
\end{array}\right.$$
Since the fact that $(w_i)_{i\ge1}$ is a $(\beta_0,\cdots,\beta_m)$-expansion of $x$ implies
$$x=\pi(w_1\cdots w_{n-1})+\frac{\pi(w_nw_{n+1}\cdots)}{\beta_{w_1}\cdots\beta_{w_{n-1}}},$$
by Lemma \ref{iteration} we know $G^{n-1}x=\pi(w_nw_{n+1}\cdots)$. This implies $G^{n-1}x\ge a_{w_n}$.
\begin{itemize}
\item[i)] If $w_n=m$, then $G^{n-1}x\ge a_m$, which implies $g_n=m=w_n$.
\item[ii)] If $w_n<m$, by condition $\pi(w_nw_{n+1}\cdots)<a_{w_n+1}$ we get $G^{n-1}x<a_{w_n+1}$. It follows from $G^{n-1}x\ge a_{w_n}$ that $g_n=w_n$.
\end{itemize}
(2) $\boxed{\Rightarrow}$ Suppose that $w$ is the quasi-greedy $(\beta_0,\cdots,\beta_m)$-expansion of $x$, i.e., $(w_i)_{i\ge1}=(g^*_i)_{i\ge1}$.
\begin{itemize}
\item[i)] Prove that $w$ does not end with $0^\infty$.
\newline(By contradiction) Assume that there exists $n\in\N$ such that $w_{n+1}w_{n+2}\cdots=0^\infty$. By Proposition \ref{expansions}, we get $(G^*)^nx=\pi(0^\infty)=0$. It follows from the definition of $G^*$ that $(G^*)^{n-1}x=0$, $(G^*)^{n-2}x=0$, $\cdots$, $G^*x=0$ and $x=0$, which contradicts $x\neq0$.
\item[ii)] Suppose $w_n<m$. Similarly to (1) $\boxed{\Rightarrow}$, we get $\pi(w_nw_{n+1}\cdots)\le a_{w_n+1}$.
\end{itemize}
$\boxed{\Leftarrow}$ follows in a way similar to (1) $\boxed{\Leftarrow}$.
\newline(3) and (4) follow in a way similar to (1) and (2) noting Lemma \ref{iff0}.
\end{proof}

\begin{proposition}[Lexicographic order on greedy, quasi-greedy, lazy and quasi-lazy expansions]\label{max-min} Let $m\in\N$, $(\beta_0,\cdots,\beta_m)\in D_m$ and $x\in[0,\frac{m}{\beta_m-1}]$.
\begin{itemize}
\item[(1)] Among all the $(\beta_0,\cdots,\beta_m)$-expansions of $x$, the greedy expansion and the lazy expansion are maximal and minimal respectively in lexicographic order.
\item[(2)] Among all the $(\beta_0,\cdots,\beta_m)$-expansions of $x$ which do not end with $0^\infty$, the quasi-greedy expansion is maximal in lexicographic order.
\item[(3)] Among all the $(\beta_0,\cdots,\beta_m)$-expansions of $x$ which do not end with $m^\infty$, the quasi-lazy expansion is minimal in lexicographic order.
\end{itemize}
\end{proposition}
\begin{proof} (1) Let $v\in\{0,\cdots,m\}^\N$ be a $(\beta_0,\cdots,\beta_m)$-expansion of $x$.
\begin{itemize}
\item[\textcircled{\footnotesize{1}}] Prove $v\preceq g(x)$.
\newline(By contradiction) Assume $v\succ g(x)$. Then there exists $n\in\N$ such that $v_1\cdots v_{n-1}=g_1\cdots g_{n-1}$ and $v_n>g_n$. Since Proposition \ref{iff} (1) implies $\pi(g_ng_{n+1}\cdots)<a_{g_n+1}$ and $(\beta_0,\cdots,\beta_m)\in D_m$ implies $a_{g_n+1}\le a_{g_n+2}\le\cdots\le a_{v_n}=\frac{v_n}{\beta_{v_n}}$,
we get $\pi(g_ng_{n+1}\cdots)<\frac{v_n}{\beta_{v_n}}$ and then
$$\begin{aligned}
x=\pi(g(x))&=\pi(g_1\cdots g_{n-1})+\frac{\pi(g_ng_{n+1}\cdots)}{\beta_{g_1}\cdots\beta_{g_{n-1}}}\\
&<\pi(v_1\cdots v_{n-1})+\frac{v_n}{\beta_{v_1}\cdots\beta_{v_{n-1}}\beta_{v_n}}\\
&=\pi(v_1\cdots v_n)\\
&\le\pi(v).
\end{aligned}$$
This contradicts $x=\pi(v)$.
\item[\textcircled{\footnotesize{2}}] We can prove $v\succeq l(x)$ in a way similar to \textcircled{\footnotesize{1}} noting that Proposition \ref{m-b expansion} implies $\frac{m}{\beta_m-1}\ge\pi(v_{n+1}v_{n+2}\cdots)$.
\end{itemize}
(2) and (3) follow in the same way as (1), noting that $v$ does not end with $0^\infty$ implies $\pi(v_1\cdots v_n)<\pi(v)$, and $v$ does not end with $m^\infty$ implies $\frac{m}{\beta_m-1}>\pi(v_{n+1}v_{n+2}\cdots)$ by Proposition \ref{m-b expansion} and Lemma \ref{iteration} for all $n\in\N$.
\end{proof}

The following definition on admissibility is a natural generalization of \cite[Definition 2.1 (2)]{LL18} (see also \cite[Definition 2.1]{LW08}).

\begin{definition}[Admissibility]\label{admissibility}
Let $m\in\N$ and $(\beta_0,\cdots,\beta_m)\in D_m$. For fixed $(I_0,\cdots,I_m)\in\cI_{\beta_0,\cdots,\beta_m}$, a sequence $w\in\{0,\cdots,m\}^\N$ is called $(I_0,\cdots,I_m)$-admissible if there exists $x\in[0,\frac{m}{\beta_m-1}]$ such that $w=t(x)$. We let $\cT=\cT(\beta_0,\cdots,\beta_m;I_0,\cdots,I_m)$ denote the set of $(I_0,\cdots,I_m)$-admissible sequences. In particular, a sequence $w\in\{0,\cdots,m\}^\N$ is called greedy, quasi-greedy, lazy and quasi-lazy (admissible) if there exists $x\in[0,\frac{m}{\beta_m-1}]$ such that $w=g(x)$, $g^*(x)$, $l(x)$ and $l^*(x)$ respectively. The sets of greedy, quasi-greedy, lazy and quasi-lazy sequences are denoted respectively by $\cG=\cG(\beta_0,\cdots,\beta_m)$, $\cG^*=\cG^*(\beta_0,\cdots,\beta_m)$, $\cL=\cL(\beta_0,\cdots,\beta_m)$ and $\cL^*=\cL^*(\beta_0,\cdots,\beta_m)$.
\end{definition}

\begin{proposition}[Commutativity]\label{commutative}
Let $m\in\N$, $(\beta_0,\cdots,\beta_m)\in D_m$ and $(I_0,\cdots,I_m)\in\cI_{\beta_0,\cdots,\beta_m}$. Then
\begin{itemize}
\item[(1)] $\pi\circ\sigma(w)=T\circ\pi(w)$ for all $w\in\cT$ and $t\circ T(x)=\sigma\circ t(x)$ for all $x\in[0,\frac{m}{\beta_m-1}]$;
\item[(2)] $\sigma(\cT)=\cT$ and $T([0,\frac{m}{\beta_m-1}])=[0,\frac{m}{\beta_m-1}]$;
\item[(3)] $t\circ\pi(w)=w$ for all $w\in\cT$ and $\pi\circ t(x)=x$ for all $x\in[0,\frac{m}{\beta_m-1}]$;
\item[(4)] $\pi|_{\cT}:\cT\to[0,\frac{m}{\beta_m-1}]$ and $t:[0,\frac{m}{\beta_m-1}]\to\cT$ are both increasing bijections.
\end{itemize}
$$\xymatrix{
\cT \ar[r]^\sigma \ar@<-0.5ex>[d]_{\pi} & \cT \ar@<-0.5ex>[d]_{\pi} \\
[0,\frac{m}{\beta_m-1}] \ar[r]^T \ar@<-0.5ex>[u]_{t}& [0,\frac{m}{\beta_m-1}] \ar@<-0.5ex>[u]_{t}
} $$
In particular, the above properties hold for the greedy, quasi-greedy, lazy and quasi-lazy cases.
\end{proposition}
\begin{proof} (1) \textcircled{\footnotesize{1}} Let $w\in\cT$. We need to prove $\pi\circ\sigma(w)=T\circ\pi(w)$. In fact, there exists $x\in[0,\frac{m}{\beta_m-1}]$ such that $w=t(x)$, and then $\pi(w)=x$ by Proposition \ref{expansions}. On the one hand,
$$\pi\circ\sigma(w)=\pi(w_2w_3\cdots)=\sum_{i=2}^\infty\frac{w_i}{\beta_{w_2}\cdots\beta_{w_i}}.$$
On the other hand,
$$T\circ\pi(w)=Tx\overset{(\star)}{=}T_{w_1}x=\beta_{w_1}x-w_1=\beta_{w_1}\sum_{i=1}^\infty\frac{w_i}{\beta_{w_1}\cdots\beta_{w_i}}-w_1=\sum_{i=2}^\infty\frac{w_i}{\beta_{w_2}\cdots\beta_{w_i}},$$
where $(\star)$ follows from the fact that $t_1(x)=w_1$ implies $x\in I_{w_1}$.
\newline\textcircled{\footnotesize{2}} Let $x\in[0,\frac{m}{\beta_m-1}]$. We need to prove $t\circ T(x)=\sigma\circ t(x)$. In fact, it follows immediately from the definition of $t$ that $t_n(Tx)=t_1(T^{n-1}(Tx))=t_1(T^nx)=t_{n+1}(x)$ for all $n\in\N$.
\newline(2) $T([0,\frac{m}{\beta_m-1}])=[0,\frac{m}{\beta_m-1}]$ follows from the definition of $T$. We prove $\sigma(\cT)=\cT$ as follows.
\newline$\boxed{\subset}$ Let $w\in\cT$. Then there exists $x\in[0,\frac{m}{\beta_m-1}]$ such that $w=t(x)$. Thus $\sigma w=\sigma\circ t(x)\xlongequal[]{\text{by (1)}}t\circ T(x)\in\cT$.
\newline$\boxed{\supset}$ Let $w\in\cT$. Then there exists $y\in[0,\frac{m}{\beta_m-1}]$ such that $w=t(y)$ and there exists $x\in[0,\frac{m}{\beta_m-1}]$ such that $y=Tx$. It follows from $w=t(y)=t(Tx)\xlongequal[]{\text{by (1)}}\sigma(t(x))$ and $t(x)\in\cT$ that $w\in\sigma(\cT)$.
\newline(3) \textcircled{\footnotesize{1}} For any $w\in\cT$, there exists $x\in[0,\frac{m}{\beta_m-1}]$ such that $w=t(x)$ and $\pi(w)=x$, which implies $t\circ\pi(w)=t(x)=w$.
\newline\textcircled{\footnotesize{2}} For any $x\in[0,\frac{m}{\beta_m-1}]$, $\pi(t(x))=x$ follows from Proposition \ref{expansions}.
\newline(4) By (3), it suffices to prove that $\pi|_\cT$ is increasing.
\newline Let $w,v\in\cT$ such that $w\prec v$. Then there exists $n\ge0$ such that $w_1\cdots w_n=v_1\cdots v_n$ and $w_{n+1}<v_{n+1}$. Let $x,y\in[0,\frac{m}{\beta_m-1}]$ such that $w=t(x)$ and $v=t(y)$. We need to prove $x<y$. In fact, by Lemma \ref{iteration} we get
\begin{eqnarray}\label{we get}
x=\pi(w_1\cdots w_n)+\frac{T^nx}{\beta_{w_1}\cdots\beta_{w_n}}\quad\text{and}\quad y=\pi(v_1\cdots v_n)+\frac{T^ny}{\beta_{v_1}\cdots\beta_{v_n}}.
\end{eqnarray}
Since $t_{n+1}(x)=w_{n+1}$ and $t_{n+1}(y)=v_{n+1}$ imply $T^nx\in I_{w_{n+1}}$ and $T^ny\in I_{v_{n+1}}$, by $w_{n+1}<v_{n+1}$ we get $T^nx<T^ny$. It follows from (\ref{we get}) and $w_1\cdots w_n=v_1\cdots v_n$ that $x<y$.
\end{proof}

The following is a generalization of \cite[Proposition 3.4]{BK07}.

\begin{proposition}[Relations between greedy/lazy and quasi-greedy/quasi-lazy expansions]\label{quasi relations} Let $m\in\N$, $(\beta_0,\cdots,\beta_m)\in D_m$ and $x\in[0,\frac{m}{\beta_m-1}]$.
\newline(1) Suppose $x\neq0$.
\begin{itemize}
\item[\textcircled{\footnotesize{1}}] $g(x)$ does not end with $0^\infty$ if and only if $g^*(x)=g(x)$.
\item[\textcircled{\footnotesize{2}}] If $g(x)$ ends with $0^\infty$, then
$$\begin{aligned}
g^*(x)&=g_1(x)\cdots g_{n-1}(x)g^*(a_{g_n(x)})\\
&=g_1(x)\cdots g_{n-1}(x)(g_n(x)-1)g^*(T_{g_n(x)-1}(a_{g_n(x)}))
\end{aligned}$$
where $n$ is the greatest integer such that $g_n(x)>0$.
\end{itemize}
(2) Suppose $x\neq\frac{m}{\beta_m-1}$.
\begin{itemize}
\item[\textcircled{\footnotesize{1}}] $l(x)$ does not end with $m^\infty$ if and only if $l^*(x)=l(x)$.
\item[\textcircled{\footnotesize{2}}] If $l(x)$ ends with $m^\infty$, then
$$\begin{aligned}
l^*(x)&=l_1(x)\cdots l_{n-1}(x)l^*(b_{l_n(x)})\\
&=l_1(x)\cdots l_{n-1}(x)(l_n(x)+1)l^*(T_{l_n(x)+1}(b_{l_n(x)}))
\end{aligned}$$
where $n$ is the greatest integer such that $l_n(x)<m$.
\end{itemize}
\end{proposition}
\begin{proof} (1) \textcircled{\footnotesize{1}} $\boxed{\Leftarrow}$ follows from Proposition \ref{iff} (2).
\newline$\boxed{\Rightarrow}$ (By contradiction) Assume $(g_i)_{i\ge1}\neq(g^*_i)_{i\ge1}$. Then there exists $n\in\N$ such that $g_1\cdots g_{n-1}=g^*_1\cdots g^*_{n-1}$ and $g_n\neq g^*_n$. Recall the definitions of $g,g^*,G$ and $G^*$. By $x\neq0$ and $g_1=g^*_1$, we get $x\notin\{a_0,\cdots,a_m\}$, and then $Gx=G^*x\neq0$. By $g_2=g^*_2$, we get $Gx=G^*x\notin\{a_0,\cdots,a_m\}$, and then $G^2x=(G^*)^2x\neq0$.$\cdots$ By repeating the above process, we get $G^{n-1}x=(G^*)^{n-1}x\neq0$. It follows from
$$G^{n-1}x\in\left\{\begin{array}{ll}
[a_{g_n},a_{g_n+1}) & \text{if }0\le g_n\le m-1,\\
\normalsize[a_m,\frac{m}{\beta_m-1}\normalsize] & \text{if }g_n=m,
\end{array}\right.$$
and $g_n\neq g^*_n$ that $G^{n-1}x=a_{g_n}$ This implies $G^nx=0$, and then for all $i\ge n$, $G^ix=0$. Thus $g_{n+1}g_{n+2}\cdots=0^\infty$, which contradicts that $(g_i)_{i\ge1}$ does not end with $0^\infty$.
\newline\textcircled{\footnotesize{2}} Suppose that $g(x)$ ends with $0^\infty$ and $n$ is the greatest integer such that $g_n>0$. We need to consider the following i), ii) and iii).
\begin{itemize}
\item[i)] Prove $g^*_1\cdots g^*_{n-1}=g_1\cdots g_{n-1}$.
\newline(By contradiction) Assume $g^*_1\cdots g^*_{n-1}\neq g_1\cdots g_{n-1}$. Then there exists $k\in\{1,\cdots,n-1\}$ such that $g^*_1\cdots g^*_{k-1}=g_1\cdots g_{k-1}$ but $g^*_k\neq g_k$. By Lemma \ref{iteration} we get $(G^*)^{k-1}x=G^{k-1}x$. Since $g^*_k\neq g_k$, there must exist $j\in\{1,\cdots,m\}$ such that $G^{k-1}x=a_j$. This implies $G^kx=0$, and then for all $i\ge k$ we have $G^ix=0$. Thus $g_{k+1}g_{k+2}\cdots=0^\infty$, which contradicts $g_n>0$.
\item[ii)] Prove $g^*_ng^*_{n+1}\cdots=g^*(a_{g_n})$. In fact, we have
$$\sigma^{n-1}(g^*(x))\overset{(\star)}{=}g^*((G^*)^{n-1}x)\overset{(\star\star)}{=}g^*(a_{g_n}),$$
where $(\star)$ follows from Proposition \ref{commutative} (1), and $(\star\star)$ follows from $(G^*)^{n-1}x=a_{g_n}$, which is a consequence of i), Lemma \ref{iteration} and
$$x=\pi(g_1\cdots g_n)=\pi(g_1\cdots g_{n-1})+\frac{a_{g_n}}{\beta_{g_1}\cdots\beta_{g_{n-1}}}.$$
\item[iii)] Prove $g^*(a_{g_n})=(g_n-1)g^*(T_{g_n-1}(a_{g_n}))$.
\newline In fact, on the one hand, $g^*_1(a_{g_n})=g_n-1$ follows directly from the definition of $g^*_1$. On the other hand, we have
$$\sigma(g^*(a_{g_n}))\overset{(\star)}{=}g^*(G^*(a_{g_n}))\overset{(\star\star)}{=}g^*(T_{g_n-1}(a_{g_n})),$$
where $(\star)$ follows from Proposition \ref{commutative} (1), and $(\star\star)$ follows from $g_n>0$ and the definition of $G^*$.
\end{itemize}
(2) follows in a way similar to (1).
\end{proof}

In the proof of our main results, we need the following.

\begin{proposition}[Interactive increase]\label{supincrease}
Let $m\in\N$, $(\beta_0,\cdots,\beta_m)\in D_m$ and $x,y\in[0,\frac{m}{\beta_m-1}]$.
\begin{itemize}
\item[(1)] Let $(I_0,\cdots,I_m),(I_0',\cdots,I_m')\in\cI_{\beta_0,\cdots,\beta_m}$ such that for all $k\in\{0,\cdots,m\}$, the intervals $I_k$ and $I_k'$ are at most different at the end points (i.e., they have the same closure), $t(x)$ be the $(I_0,\cdots,I_m)$-$(\beta_0,\cdots,\beta_m)$-expansion of $x$ and $t'(y)$ be the $(I_0',\cdots,I_m')$-$(\beta_0,\cdots,\beta_m)$-expansion of $y$. If $x<y$, then $t(x)\prec t'(y)$.
\item[(2)] In particular, if $x<y$, we have $g(x)\prec g^*(y)$ and $l^*(x)\prec l(y)$.
\end{itemize}
\end{proposition}
\begin{proof} We only need to prove (1). Suppose $0\le x<y\le\frac{m}{\beta_m-1}$. Since $t(x)=t'(y)$ will imply $x=\pi(t(x))=\pi(t'(y))=y$ which contradicts $x<y$, we must have $t(x)\neq t'(y)$. Thus there exists $n\ge0$ such that $t_1(x)\cdots t_n(x)=t'_1(y)\cdots t'_n(y)$ and $t_{n+1}(x)\neq t'_{n+1}(y)$. It suffices to prove $t_{n+1}(x)<t'_{n+1}(y)$ by contradiction.

In fact, by $x<y$ and Lemma \ref{iteration}, we get $T^nx<(T')^ny$, where $T$ is the $(I_0,\cdots,I_m)$-$(\beta_0,\cdots,\beta_m)$-transformation and $T'$ is the $(I'_0,\cdots,I'_m)$-$(\beta_0,\cdots,\beta_m)$-transformation. If $t_{n+1}(x)>t'_{n+1}(y)$, by $T^nx\in I_{t_{n+1}(x)}$ and $(T')^ny\in I'_{t'_{n+1}(y)}$ we get
$$T^nx\ge\inf I_{t_{n+1}(x)}\ge\sup I'_{t'_{n+1}(y)}\ge(T')^ny,$$
which contradicts $T^nx<(T')^ny$.
\end{proof}

Given $x\in[0,\frac{m}{\beta_m-1}]$, let
$$\Sigma_{\beta_0,\cdots,\beta_m}(x):=\Big\{(w_i)_{i\ge1}\in\{0,\cdots,m\}^\N:(w_i)_{i\ge1}\text{ is a }(\beta_0,\cdots,\beta_m)\text{-expansion of }x\Big\}$$
and
$$\Omega_{\beta_0,\cdots,\beta_m}(x):=\Big\{(S_i)_{i\ge1}\in\{T_0,\cdots,T_m\}^\N:(S_n\circ\cdots\circ S_1)(x)\in\Big[0,\frac{m}{\beta_m-1}\Big]\text{ for all }n\in\N\Big\}.$$

As a generalization of \cite[Lemma 3.1]{B20} and \cite[Lemma 2.1]{BK19} (see also \cite{B14}), the following is a dynamical interpretation of $(\beta_0,\cdots,\beta_m)$-expansions.

\begin{proposition}[Dynamical interpretation]\label{dynamical interpretation}
Let $m\in\N$ and $(\beta_0,\cdots,\beta_m)\in D_m$. For all $x\in[0,\frac{m}{\beta_m-1}]$, the map which sends $(w_i)_{i\ge1}$ to $(T_{w_i})_{i\ge1}$ is a bijection from $\Sigma_{\beta_0,\cdots,\beta_m}(x)$ to $\Omega_{\beta_0,\cdots,\beta_m}(x)$.
\end{proposition}
\begin{proof} (1) Prove that the mentioned map is well-defined.
\newline Let $(w_i)_{i\ge1}\in\{0,\cdots,m\}^\N$ be a $(\beta_0,\cdots,\beta_m)$-expansion of $x$ and $n\in\N$. It suffices to prove $T_{w_n}\circ\cdots\circ T_{w_1}x\in[0,\frac{m}{\beta_m-1}]$. In fact, by a simple calculation as in (3) in the proof of Proposition \ref{m-b expansion}, we get
$$T_{w_n}\circ\cdots\circ T_{w_1}x=\sum_{i=n+1}^\infty\frac{w_i}{\beta_{w_{n+1}}\cdots\beta_{w_i}}.$$
Thus
$$T_{w_n}\circ\cdots\circ T_{w_1}x=\sum_{i=1}^\infty\frac{w_{n+i}}{\beta_{w_{n+1}}\cdots\beta_{w_{n+i}}}=\pi(w_{n+1}w_{n+2}\cdots)\in[0,\frac{m}{\beta_m-1}]$$
by Proposition \ref{m-b expansion}.
\newline(2) The mentioned map is obviously injective. We prove that it is surjective as follows.
\newline Let $(w_i)_{i\ge1}\in\{0,\cdots,m\}^\N$ such that $T_{w_n}\circ\cdots\circ T_{w_1}x\in[0,\frac{m}{\beta_m-1}]$ for all $n\in\N$. By
$$0\le T_{w_n}\circ\cdots\circ T_{w_1}x\le\frac{m}{\beta_m-1},$$
we get
$$\frac{w_n}{\beta_{w_n}}\le T_{w_{n-1}}\circ\cdots\circ T_{w_1}x\le\frac{w_n}{\beta_{w_n}}+\frac{m}{\beta_{w_n}(\beta_m-1)},$$
$$\frac{w_{n-1}}{\beta_{w_{n-1}}}+\frac{w_n}{\beta_{w_{n-1}}\beta_{w_n}}\le T_{w_{n-2}}\circ\cdots\circ T_{w_1}x\le\frac{w_{n-1}}{\beta_{w_{n-1}}}+\frac{w_n}{\beta_{w_{n-1}}\beta_{w_n}}+\frac{m}{\beta_{w_{n-1}}\beta_{w_n}(\beta_m-1)},$$
$$\cdots,$$
$$\frac{w_1}{\beta_{w_1}}+\frac{w_2}{\beta_{w_1}\beta_{w_2}}+\cdots+\frac{w_n}{\beta_{w_1}\cdots\beta_{w_n}}\le x\le\frac{w_1}{\beta_{w_1}}+\frac{w_2}{\beta_{w_1}\beta_{w_2}}+\cdots+\frac{w_n}{\beta_{w_1}\cdots\beta_{w_n}}+\frac{m}{\beta_{w_1}\cdots\beta_{w_n}(\beta_m-1)},$$
which implies
$$\pi(w_1\cdots w_n)\le x\le\pi(w_1\cdots w_n)+\frac{m}{(\beta_m-1)(\min\{\beta_0,\cdots,\beta_m\})^n}$$
for all $n\in\N$. Let $n\to\infty$, we get $x=\pi(w_1w_2\cdots)$. Thus $(w_i)_{i\ge1}\in\Sigma_{\beta_0,\cdots,\beta_m}(x)$.
\end{proof}

The following proposition on expansions in one base, which will be used in the proof of Corollary \ref{cor3}, implies that $w$ is lazy if and only if $\overline{w}$ is greedy (recall Definition \ref{admissibility}) for all $w=(w_i)_{i\ge1}\in\{0,\cdots,m\}^\N$, where $\overline{w}:=(\overline{w_i})_{i\ge1}$ and $\overline{k}:=m-k$ for all $k\in\{0,\cdots,m\}$. By Proposition \ref{max-min} (1), we recover \cite[Theorem 2.1]{DK93} and \cite[Lemma 1]{K99}.

\begin{proposition}[Reflection principle in one base]\label{symmetric} Let $m\in\N$ and $\beta\in(1,m+1]$. For all $x\in[0,\frac{m}{\beta-1}]$, we have
$$l\Big(\frac{m}{\beta-1}-x\Big)=\overline{g(x)}\quad\text{and}\quad l^*\Big(\frac{m}{\beta-1}-x\Big)=\overline{g^*(x)}.$$
\end{proposition}
\begin{proof} (1) Prove $l(\frac{m}{\beta-1}-x)=\overline{g(x)}$. Let $w=g(x)$. By Proposition \ref{iff} (1) we get
$$\pi(w_nw_{n+1}\cdots)<a_{w_n+1}\quad\text{whenever }w_n<m.$$
It follows from $\pi(w_nw_{n+1}\cdots)+\pi(\overline{w}_n\overline{w}_{n+1}\cdots)=\frac{m}{\beta-1}$ and $a_{w_n+1}+b_{w_n-1}=\frac{m}{\beta-1}$ that
\begin{equation}\label{reflect}
\pi(\overline{w}_n\overline{w}_{n+1}\cdots)>b_{w_n-1}\quad\text{whenever }\overline{w}_n>0.
\end{equation}
Since $w=g(x)$ implies $\pi(\overline{w})=\frac{m}{\beta-1}-x$, by Proposition \ref{iff} (3) and (\ref{reflect}) we get $\overline{w}=l(\frac{m}{\beta-1}-x)$.
\newline(2) $l^*(\frac{m}{\beta-1}-x)=\overline{g^*(x)}$ follows in a way similar to (1) by applying Proposition \ref{iff} (2) and (4).
\end{proof}

\section{Proofs of the main results}

First we give the following lemma, which is essentially stronger than Theorem \ref{main} (1) \textcircled{\footnotesize{1}}, (2) \textcircled{\footnotesize{1}} and (3) \textcircled{\footnotesize{1}}.

\begin{lemma}\label{lemma}
Let $m\in\N$, $(\beta_0,\cdots,\beta_m)\in D_m$, $x\in[0,\frac{m}{\beta_m-1}]$ and $w\in\{0,\cdots,m\}^\N$ be a $(\beta_0,\cdots,\beta_m)$-expansion of $x$.
\begin{itemize}
\item[(1)] If $w$ is the greedy expansion and $w\neq m^\infty$, then
$$\sigma^nw\prec g^*(\xi_+)\quad\text{for all }n\ge p,$$
where $p:=\min\{i\ge0:G^ix<\xi_+\}$ exists, and $\xi_+:=\max_{0\le k\le m-1}T_k(a_{k+1})$.
\item[(2)] If $w$ is the lazy expansion and $w\neq0^\infty$, then
$$\sigma^nw\succ l^*(\eta_-)\quad\text{for all }n\ge q,$$
where $q:=\min\{i\ge0:L^ix>\eta_-\}$ exists, and $\eta_-:=\min_{1\le k\le m}T_k(b_{k-1})$.
\end{itemize}
\end{lemma}
\begin{proof} (1) By $(\beta_0,\cdots,\beta_m)\in D_m$, we get
$$a_k<a_{k+1}\le b_k$$
for all $k\in\{0,\cdots,m-1\}$. This implies $0<\xi_+\le\frac{m}{\beta_m-1}$.
\begin{itemize}
\item[\textcircled{\footnotesize{1}}] Prove that there exists $i\ge0$ such that $G^ix<\xi_+$.
\newline(By contradiction) Assume $G^ix\ge\xi_+$ for all $i\ge0$. Let $r$ be the greatest integer in $\{0,\cdots,m\}$ such that $a_r\le\xi_+$ and
$$c=c(x):=\left\{\begin{array}{ll}
x-\beta_mx+m& \text{if }r=m;\\
\min\{x-\beta_mx+m,a_{r+1}-\xi_+\}& \text{if } r\le m-1.
\end{array}\right.$$
It follows from $w\neq m^\infty$ (which implies $x<\frac{m}{\beta_m-1}$ by Lemma \ref{iff0}) and the definition of $r$ that $c>0$.
\begin{itemize}
\item[i)] Prove that for all $y\in[\xi_+,x]$, we have $y-Gy\ge c$.
\newline In fact, if $y\ge a_m$, then $y-Gy=y-\beta_my+m\ge x-\beta_mx+m\ge c$. We only need to consider $\xi_+\le y<a_m$ in the following. By $\xi_+<a_m$, we know $r\le m-1$ and
$$[\xi_+,a_m)\subset[a_r,a_{r+1})\cup[a_{r+1},a_{r+2})\cup\cdots\cup[a_{m-1},a_m).$$
There exists $k\in\{r,r+1,\cdots,m-1\}$ such that $y\in[a_k,a_{k+1})$. Thus
$$y-Gy=y-(\beta_ky-k)=(1-\beta_k)y+k>(1-\beta_k)a_{k+1}+k=a_{k+1}-T_k(a_{k+1})\ge a_{r+1}-\xi_+\ge c.$$
\item[ii)] Deduce a contradiction.
\newline Recall that we have assumed $G^ix\ge\xi_+$ for all $i\ge0$. First by $x\ge\xi_+$ and i), we get $x-Gx\ge c$. Then by $\xi_+\le Gx\le x$ and i) again, we get $Gx-G^2x\ge c$. $\cdots$
\newline For all $n\ge1$, we can get $G^{n-1}x-G^nx\ge c$. It follows from the summation of the above inequalities that $x-G^nx\ge nc$, where $nc\to+\infty$ as $n\to+\infty$. This contradicts $G^ix\ge \xi_+$ for all $i\ge0$.
\end{itemize}
\item[\textcircled{\footnotesize{2}}] For all $n\ge p$, $\sigma^nw\prec g^*(\xi_+)$ follows from
$$\sigma^nw=\sigma^n(g(x))\overset{(\star)}{=}g(G^nx)\overset{(\star\star)}{\prec}g^*(\xi_+),$$
where $(\star)$ follows from Proposition \ref{commutative} (1), and $(\star\star)$ follows from Proposition \ref{supincrease} and $G^nx<\xi_+$, which can be proved as follows. First we have $G^px<\xi_+$ by the definition of $p$. It suffices to prove that for all $y\in[0,\xi_+)$, we have $Gy<\xi_+$. In fact, let $y\in[0,\xi_+)\subset[0,\frac{m}{\beta_m-1})$. If $y\ge a_m$, then
$$Gy=T_my=\beta_my-m<y<\xi_+.$$
If $y<a_m$, then there exists $k\in\{0,\cdots,m-1\}$ such that $y\in[a_k,a_{k+1})$ and we have
$$Gy=T_ky<T_k(a_{k+1})\le\xi_+.$$
\end{itemize}
(2) follows in a way similar to (1) by using $a_k\le b_{k-1}<b_k$ instead of $a_k<a_{k+1}\le b_k$ for all $k\in\{1,\cdots,m\}$.
\end{proof}

\begin{proof}[Proof of Theorem \ref{main}] (1) \textcircled{\footnotesize{1}} Suppose that $w$ is the greedy $(\beta_0,\cdots,\beta_m)$-expansion of $x$ and $w_n<m$. Then $G^{n-1}x\in[a_{w_n},a_{w_n+1})$ and
$$G^nx=G(G^{n-1}x)=T_{w_n}(G^{n-1}x)<T_{w_n}(a_{w_n+1})\le\xi_+.$$
It follows from Lemma \ref{lemma} (1) that $\sigma^nw\prec g^*(\xi_+)$.
\newline\textcircled{\footnotesize{2}} Suppose $w_n<m$. By Proposition \ref{iff} (1), we only need to prove $\pi(w_nw_{n+1}\cdots)<a_{w_n+1}$, which is equivalent to $\pi(w_{n+1}w_{n+2}\cdots)<T_{w_n}(a_{w_n+1})$.

For simplification, we use $g^*_i$ to denote $g^*_i(\xi_-)$ for all $i\in\N$ in the following.

First by condition $\sigma^nw\prec g^*(\xi_-)$, we get $w_{n+1}w_{n+2}\cdots\prec g^*_1g^*_2\cdots$. Then there exist $s_1\in\N$ and $n_1=n+s_1$ such that
$$w_{n+1}\cdots w_{n_1-1}=g^*_1\cdots g^*_{s_1-1}\quad\text{and}\quad w_{n_1}<g^*_{s_1}.$$
By condition $\sigma^{n_1}w\prec g^*(\xi_-)$, we get $w_{n_1+1}w_{n_1+2}\cdots\prec g^*_1g^*_2\cdots$. Then there exist $s_2\in\N$ and $n_2=n_1+s_2$ such that
$$w_{n_1+1}\cdots w_{n_2-1}=g^*_1\cdots g^*_{s_2-1}\quad\text{and}\quad w_{n_2}<g^*_{s_2}.$$
For general $j\ge2$, if there already exist $s_j\in\N$ and $n_j=n_{j-1}+s_j$ such that
$$w_{n_{j-1}+1}\cdots w_{n_j-1}=g^*_1\cdots g^*_{s_j-1}\quad\text{and}\quad w_{n_j}<g^*_{s_j},$$
by condition $\sigma^{n_j}w\prec g^*(\xi_-)$ we get $w_{n_j+1}w_{n_j+2}\cdots\prec g^*_1g^*_2\cdots$. Then there exist $s_{j+1}\in\N$ and $n_{j+1}=n_j+s_{j+1}$ such that
$$w_{n_j+1}\cdots w_{n_{j+1}-1}=g^*_1\cdots g^*_{s_{j+1}-1}\quad\text{and}\quad w_{n_{j+1}}<g^*_{s_{j+1}}.$$
For all $i\ge1$, $s_i$ and $n_i$ are well defined by the above process. Since
$$\pi(w_{n+1}w_{n+2}\cdots)=\sum_{i=0}^\infty\frac{\pi(w_{n_i+1}\cdots w_{n_{i+1}})}{\beta_{w_{n+1}}\beta_{w_{n+2}}\cdots\beta_{w_{n_i}}}$$
and
$$T_{w_n}(a_{w_n+1})=\sum_{i=0}^\infty\Big(\frac{T_{w_{n_i}}(a_{w_{n_i}+1})}{\beta_{w_{n+1}}\beta_{w_{n+2}}\cdots\beta_{w_{n_i}}}-\frac{T_{w_{n_{i+1}}}(a_{w_{n_{i+1}}+1})}{\beta_{w_{n+1}}\beta_{w_{n+2}}\cdots\beta_{w_{n_{i+1}}}}\Big)$$
where $n_0:=n$ and $\beta_{w_{n+1}}\beta_{w_{n+2}}\cdots\beta_{w_{n_0}}:=1$, we only need to prove
$$\pi(w_{n_i+1}\cdots w_{n_{i+1}})<T_{w_{n_i}}(a_{w_{n_i}+1})-\frac{T_{w_{n_{i+1}}}(a_{w_{n_{i+1}}+1})}{\beta_{w_{n_i+1}}\beta_{w_{n_i+2}}\cdots\beta_{w_{n_{i+1}}}},$$
$$\text{i.e.,}\quad\pi(w_{n_i+1}\cdots w_{n_{i+1}-1})+\frac{a_{w_{n_{i+1}}}}{\beta_{w_{n_i+1}}\beta_{w_{n_i+2}}\cdots\beta_{w_{n_{i+1}-1}}}<T_{w_{n_i}}(a_{w_{n_i}+1})-\frac{a_{w_{n_{i+1}}+1}-a_{w_{n_{i+1}}}}{\beta_{w_{n_i+1}}\beta_{w_{n_i+2}}\cdots\beta_{w_{n_{i+1}-1}}},$$
$$\text{i.e.,}\quad\pi(w_{n_i+1}\cdots w_{n_{i+1}-1})+\frac{a_{w_{n_{i+1}}+1}}{\beta_{w_{n_i+1}}\beta_{w_{n_i+2}}\cdots\beta_{w_{n_{i+1}-1}}}<T_{w_{n_i}}(a_{w_{n_i}+1})\quad\text{for all }i\ge0.$$
In fact, for all $i\ge0$, by $w_{n_i+1}\cdots w_{n_{i+1}-1}=g^*_1\cdots g^*_{s_{i+1}-1}$ and $w_{n_{i+1}}+1\le g^*_{s_{i+1}}$ (which implies $a_{w_{n_{i+1}}+1}\le a_{g^*_{s_{i+1}}}$), we get
$$\begin{aligned}
\pi(w_{n_i+1}\cdots w_{n_{i+1}-1})+\frac{a_{w_{n_{i+1}}+1}}{\beta_{w_{n_i+1}}\beta_{w_{n_i+2}}\cdots\beta_{w_{n_{i+1}-1}}}&\le\pi(g^*_1\cdots g^*_{s_{i+1}-1})+\frac{a_{g^*_{s_{i+1}}}}{\beta_{g^*_1}\beta_{g^*_2}\cdots\beta_{g^*_{s_{i+1}-1}}}\\
&=\pi(g^*_1\cdots g^*_{s_{i+1}})\\
&\overset{(\star)}{<}\pi(g^*(\xi_-))=\xi_-\le T_{w_{n_i}}(a_{w_{n_i}+1}),
\end{aligned}$$
where $(\star)$ follows from the fact that $g^*(\xi_-)$ does not end with $0^\infty$ (by Proposition \ref{iff} (2)).
\newline(2) follows in a way similar to (1).
\newline(3) follows immediately from (1), (2) and Proposition \ref{max-min} (1).
\end{proof}

Corollary \ref{cor1} follows directly from Theorem \ref{main}.

Corollary \ref{cor2} follows from Theorem \ref{main}, the facts that $\beta_0\le\beta_1\le\cdots\le\beta_m$ implies $\xi_+\le1$ and $\eta_-\ge\frac{m}{\beta_m-1}-1$, $\beta_0\ge\beta_1\ge\cdots\ge\beta_m$ implies $\xi_-\ge1$ and $\eta_+\le\frac{m}{\beta_m-1}-1$, and the increase of $g^*$ and $l^*$ (by Proposition \ref{commutative} (4)).

\begin{proof}[Proof of Corollary \ref{cor3}] (1) follows immediately from Corollary \ref{cor2} and Proposition \ref{symmetric}.
\newline(2) \textcircled{\footnotesize{1}} $\boxed{\Rightarrow}$ follows from Lemma \ref{lemma} (1), in which $\xi_+=1$ and $p=0$.
\newline$\boxed{\Leftarrow}$ First by (1) \textcircled{\footnotesize{1}}, we know that $w$ is the greedy expansion $g(x)$. Then it follows from $g(x)=w<g^*(1)\le g(1)$ and the strictly increase of $g$ (by Proposition \ref{commutative} (4)) that $x<1$.
\newline\textcircled{\footnotesize{2}} $\boxed{\Rightarrow}$ follows from Proposition \ref{symmetric} and Lemma \ref{lemma} (2), in which $\eta_-=\frac{m}{\beta-1}-1$ and $q=0$.
\newline$\boxed{\Leftarrow}$ First by (1) \textcircled{\footnotesize{2}}, we know that $w$ is the lazy expansion $l(x)$. Then it follows from $l(x)=w>\overline{g^*(1)}=l^*(\frac{m}{\beta-1}-1)\ge l(\frac{m}{\beta-1}-1)$ and the strictly increase of $l$ (by Proposition \ref{commutative} (4)) that $x>\frac{m}{\beta-1}-1$.
\newline\textcircled{\footnotesize{3}} follows from \textcircled{\footnotesize{1}}, \textcircled{\footnotesize{2}} and Proposition \ref{max-min} (1).
\end{proof}

\section{Further questions}

On the one hand, although necessary and sufficient conditions for sequences to be greedy, lazy and unique expansions in two bases and one base are obtained in Corollary \ref{cor1} and \ref{cor3} respectively, for general cases, i.e., in more than two bases, Theorem \ref{main} and Corollary \ref{cor2} can only give necessary conditions and sufficient conditions separately. We look forward to getting necessary and sufficient conditions for general cases.

On the other hand, in our main results, including Theorem \ref{main}, Corollary \ref{cor1}, \ref{cor2} and \ref{cor3}, we can see that some special expansions of $\xi_+,\xi_-,\eta_+,\eta_-,1$ and $\frac{m}{\beta_m-1}-1$ play important roles in determining greedy, lazy and unique expansions of general $x$. Thus we think that it is meaningful to characterize the greedy, quasi-greedy, lazy, quasi-lazy and unique expansions of $\xi_+,\xi_-,\eta_+,\eta_-,1$ and $\frac{m}{\beta_m-1}-1$ in multiple bases, especially in combinatorial ways. See \cite{EJK90} for combinatorial characterizations of greedy, lazy and unique expansions of $1$ in one base.

\begin{ack}
The author thanks Professor Jean-Paul Allouche for his advices and pointing to the paper \cite{N19}. The author is also grateful to the Oversea Study Program of Guangzhou Elite Project (GEP) for financial support (JY201815).
\end{ack}

\end{document}